\documentclass{amsart}
\usepackage{amsmath}
\usepackage{amsthm}
\usepackage{amssymb}

\theoremstyle{plain}
  \newtheorem{theorem}{\bf Theorem}[section]
  \newtheorem{propn}[theorem]{\bf Proposition}
  \newtheorem{lemma}[theorem]{\bf Lemma}
  \newtheorem{Observation}[theorem]{\bf Observation}
  \newtheorem{corollary}[theorem]{\bf Corollary}

\theoremstyle{remark}
  
\theoremstyle{remarks}

\newcommand{\supp}{\mathrm{supp}}

\begin{document}
\title[Weighted inequalities for oscillatory integrals of finite type]{Weighted norm inequalities for oscillatory integrals with finite type phases on the line} \keywords{Weighted inequalities, oscillatory integrals, convolution}
\subjclass[2000]{44B20; 42B25}
\author{Jonathan Bennett}
\author{Samuel Harrison}
\thanks{The first author was partially supported by EPSRC
grant EP/E022340/1, and the second by an EPSRC Doctoral Training Grant.
}
\address{Jonathan Bennett and Samuel Harrison, School of Mathematics, The Watson Building, University of Birmingham, Edgbaston,
Birmingham, B15 2TT, England.} \email{J.Bennett@bham.ac.uk}
\email{harriss@maths.bham.ac.uk}
\begin{abstract}
We obtain two-weighted $L^2$ norm inequalities for oscillatory integral operators of convolution type on the line whose phases are of finite type. The conditions imposed on the weights involve geometrically-defined maximal functions, and the inequalities are best-possible in the sense that they imply the full $L^p(\mathbb{R})\rightarrow L^q(\mathbb{R})$ mapping properties of the oscillatory integrals. Our results build on work of Carbery, Soria, Vargas and the first author.
\end{abstract}
\maketitle
\section{Introduction}

Weighted norm inequalities have been the subject of intense study in harmonic analysis in recent decades. On an informal level, given a suitable operator $T$ (which maps functions on $\mathbb{R}^m$ to functions on $\mathbb{R}^n$, say) and an exponent $p\in [1,\infty)$, such inequalities typically take the form
\begin{equation}\label{gensetup}
\int_{\mathbb{R}^n}|Tf|^pv\leq\int_{\mathbb{R}^m}|f|^pu,
\end{equation}
where the weights $u$ and $v$ are certain nonnegative functions on $\mathbb{R}^m$ and $\mathbb{R}^n$ respectively. In such a context the general goal is to understand geometrically the pairs of weights $u$ and $v$ for which \eqref{gensetup} holds for all admissible inputs $f$. As may be expected, sufficient conditions on $u,v$ are often of the form $\mathcal{M}v\leq u$ for some appropriate maximal operator $\mathcal{M}$ capturing certain geometric characteristics of $T$. Under such circumstances a simple duality argument generally allows \eqref{gensetup} to transfer bounds on $\mathcal{M}$ to bounds on $T$. For example, if $q,\widetilde{q}\geq p$ then
\begin{eqnarray}\label{dualityarg}
\begin{aligned}
\|Tf\|_{L^q(\mathbb{R}^n)}&=\sup_{\|v\|_{(q/p)'}=1}\left(\int_{\mathbb{R}^n}|Tf|^pv\right)^{1/p}\\
&\leq \sup_{\|v\|_{(q/p)'}=1}\left(\int_{\mathbb{R}^n}|f|^p\mathcal{M}v\right)^{1/p}\\
&\leq\sup_{\|v\|_{(q/p)'}=1}\|\mathcal{M}v\|_{L^{(\widetilde{q}/p)'}(\mathbb{R}^m)}^{1/p}\|f\|_{L^{\widetilde{q}}(\mathbb{R}^m)},
\end{aligned}
\end{eqnarray}
and so $\|T\|_{\widetilde{q}\rightarrow q}\leq\|\mathcal{M}\|_{(q/p)'\rightarrow(\widetilde{q}/p)'}^{1/p}$. Thus given such an operator $T$ and an index $p$, it is of particular interest to identify a corresponding geometrically defined maximal operator $\mathcal{M}$ which is \emph{optimal} in the sense that all ``interesting" $L^q\rightarrow L^{\widetilde{q}}$ bounds for $T$ may be obtained from those of $\mathcal{M}$ in this way.\footnote{One might interpret ``interesting" here as those which generate the full mapping properties of $T$ by duality considerations and interpolation with elementary inequalities.}

This endeavour has been enormously successful for broad classes of important operators $T$ in euclidean harmonic analysis, such as maximal averaging operators, fractional integral operators, Calder\'on--Zygmund singular integral operators and square functions. For example if $T$ is a standard Calder\'on--Zygmund singular integral operator, such as the Hilbert transform on the line, C\'ordoba and Fefferman \cite{CF} showed that for any $p,r>1$, there is a constant $C_{p,r}$ for which
$$
\int_{\mathbb{R}^n}|Tf|^pw\leq C_{p,r}\int_{\mathbb{R}^n}|f|^pM_rw.
$$
Here $M_rw:=(Mw^r)^{1/r}$, where $M$ denotes the Hardy--Littlewood maximal function on $\mathbb{R}^n$.\footnote{There are improvements of the above inequality due to Wilson \cite{W} and P\'erez \cite{P1}. We shall appeal to these in Section \ref{litpal}.} The modern description of such inequalities belongs to the theory of $A_p$ weights, and may be found in many sources; see for example \cite{S1}, \cite{Duo} and \cite{GCRdF}.

It is pertinent at this stage to mention one further particular example of this perspective.
Let $I_\alpha$ denote the fractional integral operator of order $\alpha$ on $\mathbb{R}$, defined via the Fourier transform by $\widehat{I_\alpha f}(\xi)=|\xi|^{-\alpha}\widehat{f}(\xi)$. In \cite{P2} P\'erez showed that for any $0<\alpha<1$, there is a constant $C_\alpha$ for which
\begin{equation}\label{fracint}
 \int_{\mathbb{R}}|I_{\alpha}f|^2w
\leq
C_{\alpha}\
\int_{\mathbb{R}}|f|^2M_{2\alpha}M^{2}w,
\end{equation}
where now $$M_\alpha w(x) =  \sup_{r>0}\frac{1}{r^{1-\alpha}}\int_{|x-y|<r}w(y)dy$$ is a certain fractional Hardy--Littlewood maximal function and $M^2=M\circ M$ denotes the composition of the Hardy--Littlewood maximal function $M$ with itself.\footnote{In \cite{P2} a similar statement is proved in all dimensions and for exponents $1<p<\infty$.} From the current perspective the factors of $M$ in the maximal operator $M_\alpha M^2$ are of secondary importance since $M_\alpha$ and $M_\alpha M^2$ share the same $L^p\rightarrow L^q$ mapping properties provided $1<p\leq\infty$. This follows from the $L^p\rightarrow L^p$ boundedness of $M$ for $1<p\leq\infty$.

It is of course natural to seek such weighted inequalities for other classes of operators which occupy a central place in harmonic analysis. Perhaps the most apparent context would be that of \emph{oscillatory} integral operators, and this has indeed received some notable attention in the literature. In particular, in the Proceedings of the 1978 Williamstown Conference on Harmonic Analysis, Stein \cite{S}
raised the possibility that the disc multiplier operator or Bochner--Riesz multiplier operators may be controlled by Kakeya or Nikodym type maximal functions via weighted $L^2$ inequalities of the above general form. Although this question posed by Stein (which is sometimes referred to as Stein's Conjecture) remains largely unsolved in all dimensions, it has generated a number of results of a similar nature (see for example \cite{CRS}, \cite{CS1}, \cite{CS2}, \cite{BRV}, \cite{CSe}, \cite{BCSV}, \cite{BBC3} and \cite{LRS}).

The aim of this paper is to establish a general result providing weighted norm inequalities of the form \eqref{gensetup} for oscillatory integral operators of convolution type on the line with phases of finite type. Our approach builds on work of Carbery, Soria, Vargas and the first author \cite{BCSV}.

\section{Statement of results}\label{statement}
Let $\ell\in\mathbb{N}$ satisfy $\ell\geq 2$ and $x_0\in\mathbb{R}$ be given. Suppose that
$\phi:\mathbb{R}\rightarrow\mathbb{R}$ is a smooth phase function satisfying the finite type condition
\begin{eqnarray}\label{hypotheses}
\phi^{(k)}(x_0) = 0 \quad\mbox{for}\quad 2 \leq k <\ell,
\quad\mbox{and}\quad \phi^{(\ell)}(x_0)\not=0.
\end{eqnarray}
As we shall clarify in Section \ref{properties}, this condition ensures that $\phi$ is close to the model phase $x\mapsto a+bx+c(x-x_0)^\ell$ in a neighbourhood of $x_0$. Here $a,b,c$ are real numbers.

For each $\lambda>0$ define a convolution kernel $K_\lambda:\mathbb{R}\rightarrow\mathbb{C}$ by
\begin{equation}\label{kerneldef}
K_\lambda (x) = e^{i\lambda\phi(x)} \psi(x),
\end{equation}
where $\psi$ is a smooth cutoff function supported in a small neighbourhood $U$ of $x_0$. Throughout we shall suppose that $U$ is sufficiently small so that $\phi^{(\ell)}$ is bounded below by a positive constant on $U$. We define the oscillatory integral operator $T_\lambda$ by
$$T_\lambda f(x)=K_\lambda*f(x)=\int_{\mathbb{R}}e^{i\lambda\phi(x-y)}\psi(x-y)f(y)dy.$$

Our main result is the following:
\begin{theorem}\label{two weighted theorem}
There exists a constant $C>0$ such that for all weight functions $w$ and $\lambda\geq 1$,
\begin{eqnarray}\label{two weighted}
\int_\mathbb{R} |T_\lambda f(x)|^2 w(x)dx \leq C
 \int_\mathbb{R} |f(x)|^2 M^2
\mathcal{M}_{\ell, \lambda} M^4 w(x-x_0)dx
\end{eqnarray}
where $M^k$ denotes the $k$-fold composition of the Hardy-Littlewood
maximal function $M$, and $\mathcal{M}_{\ell, \lambda}$ is given by
\begin{eqnarray*}
\mathcal{M}_{\ell,\lambda} w(x) = \sup_{(y,r) \in \Gamma_{\ell, \lambda}
(x)} \frac{1}{(\lambda r)^\frac{1}{\ell-1}} \int_{|y-y'| \leq r} w(y')dy'
\end{eqnarray*}
where $\Gamma_{\ell, \lambda} (x)$ is the region
\begin{eqnarray*}
\{ (y,r) : \lambda^{-1}<r\leq \lambda^{-\frac{1}{\ell}}, \quad\mbox{and}\quad |y-x|\leq
(\lambda r)^{-\frac{1}{\ell-1}} \}.
\end{eqnarray*}
\end{theorem}

The remainder of this section consists of a number of remarks on the context and implications of the above theorem, followed by an informal description of the ideas behind our proof.

First of all it should be noted that by translation-invariance it suffices to prove \eqref{two weighted} with $x_0=0$.

We will actually prove a uniform version of the above theorem under the following quantified version of the hypothesis \eqref{hypotheses}: Suppose that $\epsilon>0$ and
$(A_j)_{j\in\mathbb{N}}$ is a sequence of positive constants. In addition to \eqref{hypotheses}, suppose that
\begin{eqnarray}\label{hypothesesquant}
\phi^{(k)}(x_0) = 0 \quad\mbox{for}\quad 2 \leq k < \ell,
\quad\mbox{and}\quad \phi^{(\ell)}(x_0)\geq\epsilon,
\end{eqnarray}
and that
\begin{equation*}
\|\phi^{(j)}\|_\infty \leq A_j
\end{equation*}
for all $j\in\mathbb{N}$.
By the Mean Value Theorem, the neighbourhood $U$ of
$x_0$ may be chosen, depending only on $\epsilon$ and $A_{\ell+1}$, such that
$\phi^{\ell} \geq \epsilon/2$ on $U$. As may be seen from the proof, the constant $C$ in Theorem \ref{two weighted theorem} depends
only on $\epsilon$, $\ell$ and finitely many of the $A_j$. As might be expected, such uniform versions allow one to handle multivariable phases which satisfy the hypotheses of the theorem in one scalar variable uniformly in the remaining variables. We do not elaborate on this.

By modulating the input $f$ and output $T_\lambda f$ appropriately in \eqref{two weighted},
there is no loss of generality in strengthening the hypothesis \eqref{hypotheses} to
\begin{eqnarray}\label{hypothesessupweaker}
\phi^{(k)}(x_0) = 0 \quad\mbox{for}\quad 0 \leq k < \ell,
\quad\mbox{and}\quad \phi^{(\ell)}(x_0)\not=0.
\end{eqnarray} Similarly, \eqref{hypothesesquant} may be replaced with
 \begin{eqnarray}\label{hypothesesquantsupweaker}
\phi^{(k)}(x_0) = 0 \quad\mbox{for}\quad 0 \leq k < \ell,
\quad\mbox{and}\quad \phi^{(\ell)}(x_0)\geq\epsilon.
\end{eqnarray}
Notice that if $\phi$ satisfies the hypotheses \eqref{hypothesessupweaker} (or \eqref{hypothesesquantsupweaker}) and $\chi$ is a local diffeomorphism in a neighbourhood of $y_0\in\mathbb{R}$ with
$\chi(y_0) = x_0$, then the phase function $\phi
\circ \chi$ satisfies the hypotheses \eqref{hypothesessupweaker} (or \eqref{hypothesesquantsupweaker}) at the point
$y_0$.

The maximal function $\mathcal{M}_{\ell, \lambda}$ is a fractional Hardy--Littlewood maximal operator corresponding to an ``approach region", somewhat reminiscent of (yet different from) the maximal operators studied by Nagel and Stein in \cite{NS}. It should be noted that $\mathcal{M}_{\ell,\lambda}$ is universal in the sense that it depends only on the parameters $\ell$ and $\lambda$, and is otherwise independent of the phase $\phi$.

Since $\mathcal{M}_{\ell,\lambda}$ involves a fractional average, \eqref{two weighted} bears some resemblance to the two-weighted inequality for the fractional integral \eqref{fracint} due to P\'erez \cite{P2}. The root of this similarity lies in the fact that the Fourier multipliers for the operators $T_\lambda$ and $I_\alpha$ both exhibit ``power-like" decay. We remark that the factors of Hardy--Littlewood maximal function $M$ are of secondary importance in Theorem \ref{two weighted theorem} since for $1<p,q\leq\infty$, $M^2
\mathcal{M}_{\ell, \lambda} M^4$ and $\mathcal{M}_{\ell, \lambda}$ share the same $L^p\rightarrow L^q$ bounds (up to absolute constants). Several of these factors of $M$ do not appear to be essential and arise for technical reasons.

The model phase functions satisfying \eqref{hypotheses} are of course $\phi(x)=(x-x_0)^\ell$ for $\ell\geq 2$, although there are others of particular interest. For example the phase function $\phi(x)=\cos x$ (which clearly satisfies our hypotheses with $\ell=2$ and $\ell=3$, depending on the point $x_0$) arises naturally in the context of weighted inequalities for the Fourier extension operator associated with the circle $\mathbb{S}^1$ in $\mathbb{R}^2$.
The Fourier extension (or adjoint Fourier restriction) operator associated with $\mathbb{S}^1$ is the map
$
g\longmapsto \widehat{gd\sigma}$, where
$$\widehat{gd\sigma}(\xi)=\int_{\mathbb{S}^1}g(x)e^{ix\cdot\xi}d\sigma(x).$$ Here $\sigma$ denotes arclength measure on $\mathbb{S}^1$ and $g\in L^1(\mathbb{S}^1)$; thus $\widehat{gd\sigma}$ is simply the Fourier transform of the singular measure $gd\sigma$ on $\mathbb{R}^2$. It should be noted that the adjoint of this map is the restriction $f\mapsto \widehat{f}\bigl|_{\mathbb{S}^1}$, where again $\widehat{\;}$ denotes the Fourier transform on $\mathbb{R}^2$.
On parametrising $\mathbb{S}^1$, invoking a suitable partition of unity and applying Theorem \ref{two weighted theorem} with $\phi(x)=\cos x$ for both $\ell=2,3$, one may deduce the following theorem of Carbery, Soria, Vargas and the first author.
\begin{theorem}[\cite{BCSV}]\label{babystein}
There exists a constant $C>0$ such that for all Borel measures $\mu$ supported on $\mathbb{S}^{1}$ and $R\geq 1$,
\begin{equation*}\label{steinspherenolog}
\int_{\mathbb{S}^{1}}|\widehat{gd\sigma}(Rx)|^{2}d\mu(x)
\leq \frac{C}{R}\int_{\mathbb{S}^{1}}|g(\omega)|^{2}
M^2\mathfrak{M}_{R}M^{4}(\mu)(\omega) d\sigma(\omega),
\end{equation*}
where
\begin{equation*}
\mathfrak{M}_{R}\mu(\omega)=\sup_{\substack{ T\parallel \omega ,  \\
R^{-1}\leq \alpha \leq R^{-2/3}}}
\frac{\mu(T(\alpha,\alpha^{2}R))}{\alpha},
\end{equation*}
and $M$ is the Hardy--Littlewood maximal function on $\mathbb{S}^{1}$.
Here $T(\alpha,\beta)$ denotes a rectangle in the plane, of short side $\alpha$
and long side $\beta$.
\end{theorem}
Our proof of Theorem \ref{two weighted theorem} builds on that of Theorem \ref{babystein} and is a testament to the robustness of the approach developed in \cite{BCSV}.

The above application of Theorem \ref{two weighted theorem} used the fact that the phase under consideration satisfied Hypothesis \eqref{hypotheses} for different values of $\ell$ in different regions. A similar approach yields two-weighted inequalities associated with polynomial phases for example. The maximal operators that feature are linear combinations of translates of the operators $\mathcal{M}_{\ell, \lambda}$, determined by the local monomial structure of the polynomial.
We do not pursue this matter further here, although it is pertinent to note that the weighted bounds on $T_\lambda f$ provided by Theorem \ref{two weighted theorem} grow with the parameter $\ell$. It is straightforward to verify that if $\ell\leq\ell'$ then $\Gamma_{\ell, \lambda}(x)\subseteq
\Gamma_{\ell', \lambda}(x)$ for $x\in\mathbb{R}$. As a consequence
$$\mathcal{M}_{\ell, \lambda}w(x)\leq \mathcal{M}_{\ell', \lambda}w(x)$$
for all $x$.

Theorem \ref{two weighted theorem} is sharp in the sense that the simple duality argument \eqref{dualityarg} when applied to \eqref{two weighted} allows the interesting
$L^p(\mathbb{R})\rightarrow L^q(\mathbb{R})$ boundedness properties of the oscillatory integral operator $T_\lambda$ to be deduced from those of the
controlling maximal function $\mathcal{M}_{\ell, \lambda}$. For $\ell=3$ the bounds on $T_\lambda$ are already
known and follow from work of Greenleaf and Seeger -- \cite{GS1}\footnote{This is only explicit for $\ell=3$, although it is clear that their techniques extend to general $\ell$ in this setting.}. The central estimates are the following, from which all others may be obtained by interpolation with elementary estimates and duality.
\begin{propn}\label{maximal bounds}
For $\ell\geq 2$, there exists a constant $C>0$ such that
\begin{eqnarray}\label{bound on M intro}
\|\mathcal{M}_{\ell, \lambda}f\|_{(\frac{\ell}{2})'} \leq C \lambda^{-\frac{2}{\ell}}\|f\|_{(\frac{\ell}{2})'}
\end{eqnarray}
holds for all $f \in L^{(\frac{\ell}{2})'}(\mathbb{R})$ and $\lambda\geq 1$.
\end{propn}
Theorem \ref{two weighted theorem} combined with Proposition \ref{maximal bounds} yields the following:
\begin{corollary}\label{recovering bounds}
For $\ell\geq 2$, there exists a constant $C>0$ such that
\begin{eqnarray}\label{corollary 2}
\|T_\lambda f \|_\ell \leq C \lambda^{-\frac{1}{\ell}} \|f\|_\ell,
\end{eqnarray}
holds for all $f \in L^\ell(\mathbb{R})$ and $\lambda\geq 1$.
\end{corollary}
As in the statement of Theorem \ref{two weighted theorem}, the constants $C$ above depend only on $\epsilon$, $\ell$ and finitely many of the $A_j$.

Theorem \ref{two weighted theorem} (and thus Corollary \ref{recovering bounds}) is only truly significant for $\ell>2$. The case $\ell=2$ of Theorem \ref{two weighted theorem} may be handled by elementary methods using the local nature of the operator $T_\lambda$, Plancherel's theorem and a simple stationary phase estimate on $\widehat{K}_\lambda$.
Notice that if $\phi(x)=x^2$, then
$$e^{i\phi}*f(x)=\int_{\mathbb{R}}e^{i(x-y)^2}f(y)dy=e^{ix^2}\int_{\mathbb{R}}e^{-2ixy}(e^{iy^2}f(y))dy,$$
and so an inequality of the form
$$\int_{\mathbb{R}}|e^{i\phi}*f(x)|^2v(x)dx\leq C \int_{\mathbb{R}}|f(x)|^2u(x)dx$$ is equivalent to
$$\int_{\mathbb{R}}|\widehat{f}(x)|^2v(x)dx\leq C \int_{\mathbb{R}}|f(x)|^2u(x)dx,$$ which is of course a weighted $L^2$ inequality for the Fourier transform. Such inequalities are known when $u$ and $v$ are certain power weights (Pitt's inequality \cite{Pitt}), and generalisations thereof involving rearrangement-invariant conditions on the weights \cite{Muck}.

In the situation where the phase $\phi$ is homogeneous (that is, when $\phi(x)=x^\ell$ for some $\ell\geq 2$), a scaling and limiting argument allows one to pass from
the local inequality \eqref{two weighted} to the global inequality
\begin{eqnarray}\label{scaled two weighted}
\int_{\mathbb{R}} |e^{i(\cdot)^\ell} \ast f(x)|^2 dw(x) \leq C \int_{\mathbb{R}} |f(x)|^2 M^2 \widetilde{\mathcal{M}}_{\ell}M^4w(x) dx,
\end{eqnarray}
where the maximal function $\widetilde{\mathcal{M}}_\ell$ is given by
\begin{eqnarray*}
\widetilde{\mathcal{M}}_{\ell}w(x) = \sup_{(y,r)\in\Gamma_\ell(x)}\frac{1}{r^\frac{1}{\ell-1}} \int_{|y-y'| \leq r} w(y')dy',
\end{eqnarray*}
and
$$\Gamma_\ell(x)=\{ (y,r) : 0<r\leq 1, \quad\mbox{and}\quad |y-x|\leq
r^{-\frac{1}{\ell-1}} \}.
$$
Notice that \eqref{scaled two weighted} is only significant for $\ell>2$, since $\widetilde{\mathcal{M}}_{2}w \equiv 2\|w\|_\infty$.

\subsection*{The ideas behind the proof of Theorem \ref{two weighted theorem}}
Our proof of Theorem \ref{two weighted theorem} relies heavily on the convolution structure of the operator $T_\lambda$, with the Fourier transform playing a central role. Our strategy will involve decomposing the support of the Fourier transform of the input function $f$. It is therefore appropriate that we begin with the following elementary observation, which is a simple manifestation of the uncertainty principle.
\begin{Observation}\label{elobs}
Suppose $f\in L^1(\mathbb{R})$ is such that the support of $\widehat{f}$ is contained in a bounded subset $I\subset\mathbb{R}$, and choose $\Psi\in\mathcal{S}(\mathbb{R})$ such that $\widehat{\Psi}(\xi)=1$ for $\xi\in I$. Then
\begin{equation}\label{mol}\int |T_\lambda f|^2w\leq \|\Psi\|_1\int |T_\lambda f|^2|\Psi|*w\end{equation}
and
\begin{equation}\label{mol2}\int |T_\lambda f|^2w\leq \|T_\lambda\Psi\|_1\int |f|^2|T_\lambda\Psi|*w.\end{equation}
\end{Observation}
Observation \ref{elobs} is a simple consequence of the identities $$T_\lambda f=\Psi*(T_\lambda f)=(T_\lambda \Psi)*f$$ combined with applications of the Cauchy--Schwarz inequality and Fubini's theorem.

Since we wish to use \eqref{mol2} to prove \eqref{two weighted}, there are two questions that we must address:
\begin{itemize}
\item[(i)] How do we decompose frequency space so that the resulting functions $|T_\lambda\Psi|$ have a clear geometric interpretation?
\item[(ii)] How do we then find sufficient (almost) orthogonality on $L^2(w)$ to allow us to put the pieces of the decomposition back together again?
\end{itemize}
The frequency decomposition that we employ comes in two stages. The first stage involves the use of fairly classical Littlewood--Paley theory to reduce to the situation where the support of $\widehat{f}$ is contained in a dyadic interval. This is natural as the multiplier $\widehat{K}_\lambda(\xi)$ has power-like decay as $|\xi|\rightarrow\infty$. This Fourier support restriction allows us to mollify the weight function $w$ via \eqref{mol}, and accounts for the integration in the definition of the maximal operator $\mathcal{M}_{\ell,\lambda}$. Since we wish to ultimately apply \eqref{mol2}, it is necessary for us to decompose the support of $\widehat{f}$ further, this time into intervals of equal length. A stationary phase argument then reveals that provided the scale of this finer frequency decomposition is sufficiently small, the corresponding objects $|T_\lambda\Psi|$ have a clear geometric interpretation (satisfying estimates similar to those of $\Psi$), allowing us to appeal to \eqref{mol2}.

This analysis provides us with weighted norm inequalities for $T_\lambda$ acting on functions $f$ with very particular Fourier supports. However, a decomposition of a general function $f$ (with unrestricted Fourier support) into such functions will not in general exhibit any (almost) orthogonality on $L^2(w)$. In order to overcome this obstacle we find an efficient way to dominate the weight $w$ by a further weight $w'$ which is sufficiently smooth for our decomposition to be almost orthogonal on $L^2(w')$. The process by which we pass to this larger weight involves a local supremum, and accounts for the presence of the approach region $\Gamma_{\ell,\lambda}$ in the definition of $\mathcal{M}_{\ell,\lambda}$.

This paper is organised as follows. In Section \ref{properties} we make some simple reductions and observations pertaining to the class of phase functions $\phi$ that we consider. In Section \ref{litpal} we establish the Littlewood--Paley theory that we shall need in the proof of Theorem \ref{two weighted theorem}. The proof of Theorem \ref{two weighted theorem} is provided in Sections \ref{LPreduction}, \ref{mainsec} and \ref{mainsecinf}, and the proof of Proposition \ref{maximal bounds} in Section \ref{corsection}.

\section{Properties of the phase}\label{properties}
In this section we recall and further develop the properties of the phase $\phi$ introduced in Section \ref{statement}.

Let $\epsilon>0$ and let $(A_j)_{j=0}^\infty$ be a sequence of positive real numbers.
An elementary calculation reveals that there is a constant $C_\epsilon$, depending on at most $\epsilon$ such that the pointwise inequality $$\mathcal{M}_{\ell,\epsilon\lambda}w\leq C_\epsilon\mathcal{M}_{\ell,\lambda}w$$ holds for all weight functions $w$. As a result we may assume without loss of generality that $\epsilon=1$ in \eqref{hypothesesquantsupweaker}.

Assuming, as we may, that $x_0=0$, the hypothesis \eqref{hypothesesquantsupweaker} on the phase $\phi$ becomes
\begin{equation}\label{hypothesesquantsupweaker0}
\phi^{(k)}(0) = 0 \quad\mbox{for}\quad 0 \leq k < \ell,
\quad\mbox{and}\quad \phi^{(\ell)}(0)\geq 1.
\end{equation}
For uniformity purposes we also assume that
\begin{equation}\label{upperboundsj}
\|\phi^{(j)}\|_\infty\leq A_j
\end{equation}
for each $j\in\mathbb{N}_0$. By the mean value theorem we may of course choose a neighbourhood $U$ of the origin, depending only on $A_{\ell+1}$, such that $\phi^\ell\geq 1/2$ on $U$, and insist that the cutoff function $\psi$ in \eqref{kerneldef} is supported in $U$.

Our final observation clarifies the sense in which $x\mapsto x^\ell$ is a model for the phase function $\phi$.
If $0 \leq k \leq \ell-1$, then by Taylor's theorem, for each fixed
$x$ we have
\begin{eqnarray*}
\phi^{(k)} (x) &=& \phi^{(k)}(0) + x\phi^{(k+1)}(0) + \dots +
x^{\ell-k}\phi^{(\ell)}(y_{x,k}) \\ &=&
x^{\ell-k}\phi^{(\ell)}(y_{x,k})
\end{eqnarray*}
for some $y_{x,k} \in (0,x)$.
Since $1/2\leq |\phi^{(\ell)}|\leq A_\ell$ on the support of $\psi$,
we have
\begin{equation}\label{upperlower}
\frac{1}{2}|x|^{\ell-k}\leq |\phi^{(k)}(x)|\leq A_\ell |x|^{\ell-k}
\end{equation}
for all $x$ in the support of $\psi$ and $0\leq k\leq \ell-1$.

\subsubsection*{Notation} Throughout this paper we shall write $A\lesssim B$ if there exists a constant $c$, possibly depending on finitely many of the $A_j$, such that $A\leq cB$. In particular, this constant will always be independent of $\lambda$, the function $f$ and weight function $w$.

\section{Some Weighted Littlewood-Paley Theory}\label{litpal}

In this section we collect together the weighted inequalities for Littlewood--Paley square functions that we will appeal to in our proof of Theorem \ref{two weighted theorem}. These results, although very classical in nature, do not appear to be readily available in the literature, and hence may be of some independent interest.

\subsection{Weighted inequalities for a dyadic square function}
Let
$Q:\mathbb{R}\rightarrow\mathbb{R}$ be a smooth function such that $\widehat{Q}$ is equal
to $1$ on $[-2,-1]\cup [1,2]$, vanishing outside $[-3, -\tfrac{3}{4}]\cup [\tfrac{3}{4},3]$, and such that
$$\sum_{k\in\mathbb{Z}}\widehat{Q}(2^{-k}x)=1$$ for all $x\not=0$.
For
each $k\in\mathbb{Z}$ let
the operator $\Delta_{k}$ be given by
$$
\widehat{\Delta_{k}f}(\xi)=\widehat{Q}(2^{-k}\xi)\widehat{f}(\xi),
$$
and define the
square function $S$ by
$$
Sf(x)=\left(\sum_{k}|\Delta_{k}f(x)|^{2}\right)^{1/2}.
$$
\begin{propn}\label{DyLP} For each weight $w$ on $\mathbb{R}$,
\begin{eqnarray}\label{forward}
\int_\mathbb{R} (Sf)^2 w \lesssim \int_\mathbb{R}
|f|^2 Mw
\end{eqnarray}
and
\begin{eqnarray}\label{backwards}
\int_\mathbb{R} |f|^2 w \lesssim \int_\mathbb{R}
(Sf)^2 M^3w,
\end{eqnarray}
where $M$ denotes the Hardy--Littlewood maximal function.
\end{propn}
\begin{proof}
The forward inequality \eqref{forward} is a straightforward consequence of a more general result of Wilson \cite{W2}.

The reverse inequality \eqref{backwards} may be reduced to a well-known weighted inequality for Calder\'on--Zygmund singular integrals due to Wilson \cite{W} and P\'erez \cite{P1} as follows.
For each $j=0,1,2,3$ let
$$T_j=\sum_{k\in 4\mathbb{Z}+\{j\}}\Delta_{k},$$
so that
$f=T_0f+T_1f+T_2f+T_3f$.

Let $\epsilon=(\epsilon_k)$ be a random sequence with $\epsilon_k\in\{-1,1\}$ for each $j\in\mathbb{Z}$, and define
$$T_j^\epsilon=\sum_{k\in 4\mathbb{Z}+\{j\}}\epsilon_k\Delta_{k}$$
for each $j=0,1,2,3$.

Now let $Q':\mathbb{R}\rightarrow\mathbb{R}$ be a smooth function such that $\widehat{Q'}$ is equal
to $1$ on the support of $\widehat{Q}$ and vanishing outside $[-4, -\tfrac{1}{2}]\cup [\tfrac{1}{2},4]$, and define the operator $\Delta_k'$ by
$$
\widehat{\Delta_{k}'f}(\xi)=\widehat{Q'}(2^{-k}\xi)\widehat{f}(\xi).
$$

Observe that if we set $$S_j^\epsilon=\sum_{k\in 4\mathbb{Z}+\{j\}}\epsilon_k\Delta_k',$$ then
$$
S_j^\epsilon T_j^\epsilon=T_j,
$$
since for each $j=0,1,2,3$, the supports of $\widehat{Q}$ and $\widehat{Q'}$ ensure that $\Delta'_{k'}\Delta_k=0$ whenever $k'$ and $k$ are distinct in $4\mathbb{Z}+\{j\}$.
Hence
$$
\int_{\mathbb{R}}|f|^2w\lesssim\sum_{j=0}^3\int_{\mathbb{R}}|T_jf|^2w
=\sum_{j=0}^3\int_{\mathbb{R}}|S_j^{\epsilon}T_j^\epsilon f|^2w.
$$

It is well-known that the operator $S_j^{\epsilon}$ is a standard Calder\'on--Zygmund singular integral operator uniformly in the sequence $\epsilon$, and so by \cite{W} and \cite{P1} we have that for each $1<p<\infty$,
$$\int_{\mathbb{R}}|S^\epsilon_jf|^pw\lesssim\int_{\mathbb{R}}|f|^pM^{[p]+1}w$$
uniformly in $\epsilon$ and $j=0,1,2,3$.
Thus in particular
$$
\int_{\mathbb{R}}|f|^2w\lesssim
\sum_{j=0}^3\int_{\mathbb{R}}|T_j^\epsilon f|^2M^3w
$$
uniformly in $\epsilon$. Inequality \eqref{backwards} now follows on taking expectations and using Khinchine's inequality.
\end{proof}
We remark that the above proof also yields the weighted $L^p$ inequality
\begin{eqnarray}\label{backwardsp}
\int_\mathbb{R} |f|^p w \lesssim \int_\mathbb{R}
(Sf)^p M^{[p]+1}w
\end{eqnarray}
for all $1<p<\infty$. It would be interesting to determine whether the power $[p]+1$ of the Hardy--Littlewood maximal operator may be reduced here.

\subsection{Weighted inequalities for an ``equally-spaced" square function}
Our second square function is associated with an equally-spaced frequency decomposition. The following result is a consequence of work of Rubio de Francia \cite{Rubio}.
\begin{propn}\label{LP 1}
For $L>0$, let $W_L$ be a function on $\mathbb{R}$ with $\supp\:
\widehat{W}_L \subset \{ x \in \mathbb{R} : |x| \leq 2L \}$, such
that
\begin{eqnarray*}
\sum_{k \in \mathbb{Z}} \widehat{W}_L (x-kL) = 1
\end{eqnarray*}
for all $x \in \mathbb{R}$. Suppose further that for each $N\in\mathbb{N}$,
\begin{eqnarray*}
|W_L (x)| \lesssim \frac{L}{(1+ L|x|)^N}
\end{eqnarray*}
for all $x \in \mathbb{R}$.
For a function $f$ on $\mathbb{R}$ and $k\in\mathbb{Z}$, define $f_k$ by $\widehat{f}_k(\xi) = \widehat{f}(\xi)\widehat{W}_L(\xi - kL)$. Then for any weight function $w$ on $\mathbb{R}$,
\begin{equation*}
\int_\mathbb{R} \sum_{k\in \mathbb{Z}} |f_k|^2 w \lesssim \int_\mathbb{R} |f|^2
|W_L| \ast w.
\end{equation*}
\end{propn}
\begin{proof}
Observe that
\begin{equation*}
f_k(x) = e^{2\pi i kLx} (f(\cdot)W_L (x-
\cdot))\widehat{\hspace{2mm}} (kL),
\end{equation*}
and so
\begin{eqnarray*}
\begin{aligned}
\sum_k |f_k(x)|^2 &= \sum_k
|(f(\cdot)W_L(x-\cdot))\widehat{\hspace{2mm}}(kL)|^2 \\ &= L \int_0^{1/L} \left| \sum_k e^{2\pi i kLy} (f(\cdot)W_L(x-\cdot))\widehat{\hspace{2mm}}(kL) \right|^2 dy,
\end{aligned}
\end{eqnarray*}
by Plancherel's theorem.

By the Poisson Summation formula,
\begin{eqnarray*}
\sum_k e^{2\pi i kLy} (f(\cdot)W_L(x-\cdot))\widehat{\hspace{2mm}}(kL) =\frac{1}{L} \sum_k f(y+k/L) W_L(x-y-k/L),
\end{eqnarray*}
and so
\begin{eqnarray*}
\begin{aligned}
\sum_k |f_k(x)|^2 &= \frac{1}{L} \int_0^{1/L} \left| \sum_k
f(y+k/L)W_L(x-y-k/L) \right|^2 dy \\ &\leq \frac{1}{L}
\int_{0}^{1/L} \sum_k |f(y+k/L)|^2|W_L(x-y-k/L)| \sum_{k'}
|W_L(x-y-k'/L)|dy \\ &= \frac{1}{L} \sum_k \int_{k/L}^{(k+1)/L}
\left( \sum_{k'}|W_L(x-z-(k-k')/L)| \right) |f(z)|^2 |W_L(x-z)| dz \\
&\leq\sup_{z'\in\mathbb{R}}\sum_{k'}|W_L(z'-k'/L)|\int_\mathbb{R} |f(z)|^2 |W_L(x-z)|dz \\ &\lesssim |f|^2 \ast
|W_L|(x).
\end{aligned}
\end{eqnarray*}
\end{proof}

\section{An initial Littlewood--Paley reduction}\label{LPreduction}
In this section we use the dyadic Littlewood--Paley theory from the previous section to reduce the proof of Theorem \ref{two weighted theorem} to a weighted inequality for functions $f$ with restricted Fourier support. The nature of the frequency restrictions is motivated by the following estimates relating to the Fourier transform of
$K_\lambda$.

\begin{lemma}\label{asymptotics k hat}
\begin{eqnarray}\label{supdecay}
\sup_{y\in\mathbb{R}}\left|\int_{-\infty}^yK_\lambda(x)e^{-ix\xi}dx\right| \lesssim \left\{
\begin{array}{c}
\lambda^{-\frac{1}{\ell}}, \quad |\xi| \leq \lambda^{\frac{1}{\ell}} \\
\lambda^{-\frac{1}{2(\ell-1)}}|\xi|^{-\frac{\ell-2}{2(\ell-1)}}, \quad |\xi| > \lambda^{\frac{1}{\ell}}
\end{array} \right.
\end{eqnarray}
Moreover, for each $k,N\in\mathbb{N}$,
\begin{equation}\label{rapiddecay}\Bigl|\Bigl(\frac{d}{d\xi}\Bigr)^k\widehat{K}_\lambda(\xi)\Bigr|\lesssim |\xi|^{-N}\end{equation} for all $|\xi|\geq 2A_1\lambda$.
The implicit constants above depend on $k$, $N$, $\ell$ and finitely many of the $A_j$.
\end{lemma}
\begin{proof}
The two estimates \eqref{supdecay} follow from corresponding estimates
on the integral
\begin{eqnarray*}
\int_I e^{i(\lambda\phi(x) - x\xi)} dx
\end{eqnarray*}
that are uniform in $I$, where $I$ is an interval contained in
the support of $\psi$. This shall be achieved by routine applications of van der Corput's lemma.

Writing $h(x)=\lambda\phi(x)-x\xi$ we see that $h^{(\ell)}(x)=\lambda\phi^{(\ell)}(x)\geq\lambda/2$ for all $x\in I$ and so
\begin{eqnarray*}
\left|\int_I e^{i(\lambda\phi(x) - x\xi)} dx \right| \lesssim \lambda^{-\frac{1}{\ell}}
\end{eqnarray*}
for all $\xi\in\mathbb{R}$, with implicit constant depending only on $\ell$.

For the second estimate, let $I' = \{ x \in I : |x| \lesssim
|\xi/\lambda|^{\frac{1}{\ell-1}}\}$, with suitably small implicit constant depending on $A_\ell$.
By \eqref{upperlower}, $|h'(x)|\gtrsim |\xi|$ for all $x\in I'$, and so
\begin{eqnarray*}
\left|\int_{I'} e^{i(\lambda\phi(x) - x\xi)} dx \right| \lesssim
|\xi|^{-1}.
\end{eqnarray*}
For $x\in I\backslash I'$ we have
$|h''(x)|\gtrsim \lambda^{\frac{1}{\ell-1}}\xi^{\frac{\ell-2}{\ell-1}}$, so that
\begin{eqnarray*}
\left|\int_{I\backslash I'} e^{i(\lambda\phi(x) - x\xi)} dx \right| \lesssim
\lambda^{-\frac{1}{2(\ell-1)}} |\xi|^{-\frac{\ell-2}{2(\ell-1)}}.
\end{eqnarray*}
Now, if
$\lambda^{1/\ell} \leq |\xi|$ we have $|\xi|^{-1} \leq
\lambda^{-\frac{1}{2(\ell-1)}} |\xi|^{-\frac{\ell-2}{2(\ell-1)}}$, and so
the second estimate in \eqref{supdecay} is complete.

In order to establish \eqref{rapiddecay} observe that if $|\xi| \geq 2A_1\lambda$ then the phase $h$ has no stationary points and moreover
$|h'(x)|\gtrsim |\xi|$ uniformly in $x$. Inequality \eqref{rapiddecay} now follows by repeated integration by parts.
\end{proof}

Given Lemma \ref{asymptotics k hat} it is natural to define the sets $(\mathcal{A}_p)_{p=0}^\infty$ by
$$
\mathcal{A}_0= \{ \xi \in \mathbb{R} : |\xi| \leq
\lambda^{1/\ell}\}
$$
and
$$
\mathcal{A}_p = \{ \xi \in \mathbb{R} : 2^{p-3}\lambda^{1/\ell}<|\xi| \leq 2^{p+1}\lambda^{1/\ell} \}
$$
for $p\geq 1$.
By construction the sets $\mathcal{A}_p$ cover $\mathbb{R}$ with bounded multiplicity.
\begin{propn}\label{frequency restricted}
If the support of $\widehat{f}$ is contained in $\mathcal{A}_p$ then for all weight functions $w$ and $\lambda\geq 1$,
\begin{equation}\label{two weighted restricted}
\int_\mathbb{R} |T_\lambda f(x)|^2 w(x)dx \lesssim
 \int_\mathbb{R} |f(x)|^2 M
\mathcal{M}_{\ell, \lambda} M w(x)dx
\end{equation}
uniformly in $p$, where
$\mathcal{M}_{\ell, \lambda}$ is as in the statement of Theorem \ref{two weighted theorem}.
\end{propn}
Before we come to the proof of Proposition \ref{frequency restricted} we show how it implies Theorem \ref{two weighted theorem}.

Suppose $f:\mathbb{R}\rightarrow\mathbb{C}$ has unrestricted Fourier support and observe that for each $k\in\mathbb{Z}$ the Fourier transform of $\Delta_k f$ is supported in $\mathcal{A}_p$ for some $p$. Thus, by Proposition \ref{DyLP}, Proposition \ref{frequency restricted}, followed by Proposition \ref{DyLP} again,
\begin{eqnarray*}
\begin{aligned}
\int_{\mathbb{R}}|T_\lambda f(x)|^2w(x)dx&\lesssim\int_{\mathbb{R}}|ST_\lambda f(x)|^2M^3w(x)dx\\
&=\int_{\mathbb{R}}\sum_{k\in\mathbb{Z}} |\Delta_kT_\lambda f|^2M^3w(x)dx\\
&=\int_{\mathbb{R}}\sum_{k\in\mathbb{Z}} |T_\lambda\Delta_jf|^2M^3w(x)dx\\
&\lesssim \int_{\mathbb{R}}\sum_{k\in\mathbb{Z}} |\Delta_j f|^2M\mathcal{M}_{\ell,\lambda}M^4w(x)dx\\
&=\int_{\mathbb{R}}(Sf)^2M\mathcal{M}_{\ell,\lambda}M^4w(x)dx\\
&\lesssim \int_{\mathbb{R}}|f|^2M^2\mathcal{M}_{\ell,\lambda}M^4w(x)dx,
\end{aligned}
\end{eqnarray*}
which is the conclusion of Theorem \ref{two weighted theorem}.

It thus remains to prove Proposition \ref{frequency restricted}. By the definition of the sets $\mathcal{A}_p$ and Lemma \ref{asymptotics k hat} we have that
$$|\widehat{K}_\lambda(\xi)|\lesssim \lambda^{-\frac{1}{\ell}}2^{-\frac{p(\ell-2)}{2(\ell-1)}}$$ for all $\xi\in\mathcal{A}_p$ with $1\leq 2^p<4A_1\lambda^{(\ell-1)/\ell}$, and for each $N\in\mathbb{N}$, \begin{equation}\label{rapidap}|\widehat{K}_\lambda(\xi)|\lesssim |\xi|^{-N}\end{equation} for all $\xi\in\mathcal{A}_p$ with $2^p\geq 4A_1\lambda^{(\ell-1)/\ell}$. We therefore divide the proof of Proposition \ref{frequency restricted} into two cases. Section \ref{mainsec} is devoted to the more interesting case $1\leq 2^p<4A_1\lambda^{(\ell-1)/\ell}$, while Section \ref{mainsecinf} handles the ``error terms" corresponding to the case $2^p\geq 4A_1\lambda^{(\ell-1)/\ell}$.

\section{The proof of Proposition \ref{frequency restricted} for $1\leq 2^p<4A_1\lambda^{(\ell-1)/\ell}$}\label{mainsec}
Suppose that the function $f$ has frequencies support
in $\mathcal{A}_p$ for some $p$ with $1< 2^p<4A_1\lambda^{(\ell-1)/\ell}$; we shall deal with the case $p=0$ separately. Now, if we choose a nonnegative $\Phi\in\mathcal{S}(\mathbb{R})$ such that $\widehat{\Phi}(\xi)=1$ whenever $|\xi|\lesssim 1$, and define $\Phi_{2^p\lambda^{1/\ell}}\in\mathcal{S}(\mathbb{R})$ by $\widehat{\Phi}_{2^p\lambda^{1/\ell}}(\xi)=
\widehat{\Phi}(2^{-p}\lambda^{-1/\ell}\xi)$, then $\widehat{\Phi}_{2^p\lambda^{1/\ell}}(\xi)=1$ for all $\xi\in\mathcal{A}_p$. Hence by \eqref{mol} of Observation \ref{elobs},
\begin{equation}\label{smoothed measure}
\int_{\mathbb{R}} |T_\lambda f(x)|^2 w(x)dx
\leq \int_{\mathbb{R}}|T_\lambda f(x)|^2 w_1(x)dx,
\end{equation}
where $w_1:=\Phi_{2^p\lambda^{1/\ell}} \ast w$.

For general $p$ we have no clear geometric control of the function $|T_\lambda\Phi_{2^p\lambda^{1/\ell}}|$, and so are not in a position to meaningfully apply \eqref{mol2} of Observation \eqref{elobs} at this stage. We proceed by performing a further frequency decomposition of the function $f$ at a scale $0<L\lesssim 2^p\lambda^{1/\ell}$ to be specified later.

Let
$W\in\mathcal{S}(\mathbb{R})$ be such that $\widehat{W}$ is supported in $[-2,2]$ and
$$\sum_{k\in\mathbb{Z}}\widehat{W}(\xi-k)=1$$ for all $\xi\in\mathbb{R}$.
Define $W_{L}\in\mathcal{S}(\mathbb{R})$ by
$\widehat{W}_{L}(\xi)=\widehat{W}(\xi/L)$, and $W_{L,k}\in\mathcal{S}(\mathbb{R})$ by
$\widehat{W}_{L,k}(\xi)=\widehat{W}_L(\xi-kL)$,
so that
$$\sum_{k\in\mathbb{Z}}\widehat{W}_{L,k}(\xi)=1$$
for all $\xi\in\mathbb{R}$. Hence if $f_k=
W_{L,k}*f$ then
$$f=\sum_{k\in\mathbb{Z}}f_k,$$ and so by
the linearity of $T_\lambda$,
\begin{equation}\label{clever}
T_\lambda f=\sum_{k\in\mathbb{Z}}T_\lambda f_k.
\end{equation}
We note that since $\supp \;\widehat{f}\subseteq\mathcal{A}_p$, the only nonzero contributions to the above sum occur when $|k|\sim 2^p\lambda^{1/\ell}/L$.

Unfortunately the decomposition \eqref{clever} does not in general exhibit any (almost) orthogonality on $L^2(w_1)$ as the Fourier support of $w_1$ is too large. We thus seek an efficient way of dominating $w_1$ by a further weight which does have a sufficiently small Fourier support. This we achieve in two stages; the first involving a local supremum, and the second a carefully considered mollification. For each $x\in\mathbb{R}$ let
$$w_{2}(x)=\sup_{|x'-x|\leq (4A_1)^{-\frac{1}{\ell-1}}/L}w_1(x').$$
The factor of $4A_1$ appearing in the definition of $w_2$ is included for technical reasons that will become clear later.

Let $\Theta\in\mathcal{S}(\mathbb{R})$ be a nonnegative function whose Fourier transform is nonnegative and supported in $[-1,1]$.
Now let $$w_{3}=\Theta_L*w_2$$ where $\Theta_L\in\mathcal{S}(\mathbb{R})$ is defined by $\widehat{\Theta}_L(\xi)=\widehat{\Theta}(\xi/L)$. By construction $w_3$ has Fourier support in $[-L,L]$.
\begin{lemma}\label{forcing orthogonality}
$w_1\leq w_2 \lesssim w_3$.
\end{lemma}
\begin{proof}
The first inequality is trivial and so we focus on the second. Observe that since $\widehat{\Theta}$ is nonnegative $\Theta(0)>0$ and so by continuity there exists an absolute constant $0<c\leq 1$ such that $\Theta(x)\gtrsim 1$ whenever $|x|\leq c$.
Thus $\Theta_L(x)\gtrsim L$ whenever $|x|\leq c/L$.
Consequently
\begin{eqnarray*}
w_3(x)=\Theta_L \ast w_2 (x) &\gtrsim& L \int_{|x'| \leq
c/L} w_2(x-x')dx' \\ &\geq& L\int_{|x'| \leq
\tilde{c}/L} w_2(x-x')dx'
\end{eqnarray*}
where $\tilde{c} = \min\{c, (4A_1)^{-1/(\ell-1)}\}$.
By the definition of $w_2$ and elementary considerations, either
\begin{eqnarray*}
w_2(x-x') \geq w_2(x) \quad\mbox{for all}\quad - (4A_1)^{-\frac{1}{\ell-1}}/L \leq x' \leq 0,
\end{eqnarray*}
or
\begin{eqnarray*}
w_2(x-x') \geq w_2(x) \quad\mbox{for all}\quad 0
\leq x' \leq (4A_1)^{-\frac{1}{\ell-1}}/L,
\end{eqnarray*}
and so $w_3(x) \gtrsim w_2(x)$
with implicit constant depending only on $A_1$ and $\ell$.
\end{proof}
We now see that the decomposition \eqref{clever} is almost orthogonal in the smaller $L^2(w_3)$.
By \eqref{smoothed measure}, Lemma \ref{forcing orthogonality}, \eqref{clever} and Parseval's identity we have
\begin{eqnarray*}
\begin{aligned}
\int_{\mathbb{R}}|T_\lambda f|^2w&\leq \int_{\mathbb{R}}|T_\lambda f|^2w_3\\
&=\int_{\mathbb{R}}\left|\sum_{k\in\mathbb{Z}}T_\lambda f_k\right|^2w_3\\
&=\int_{\mathbb{R}}\sum_{k,k'\in\mathbb{R}}T_\lambda f_k\: \overline{T_\lambda f_{k'}}\:w_3\\
&=\int_{\mathbb{R}}\sum_{k,k'\in\mathbb{R}}\widehat{T_\lambda f_k}*\widehat{\widetilde{T_\lambda f_{k'}}}\:\widehat{w}_3.
\end{aligned}
\end{eqnarray*}
Since $\supp\;\widehat{f}_k\subset [(k-2)L, (k+2)L]$ we have
$$
\supp\;\widehat{T_\lambda f_k}*\widehat{\widetilde{T_\lambda f_{k'}}}\subseteq [(k-k'-4)L, (k-k'+4)L],$$
and so
$$\widehat{T_\lambda f_k}*\widehat{\widetilde{T_\lambda f_{k'}}}\:\widehat{w}_3\equiv 0$$
whenever $|k-k'|\geq 6$. Thus by the Cauchy--Schwarz inequality we have
\begin{eqnarray}\label{gotorth}
\begin{aligned}
\int_{\mathbb{R}}|T_\lambda f|^2w&\lesssim \int_{\mathbb{R}}\sum_{k,k'\in\mathbb{R}: |k-k'|<6}T_\lambda f_k\: \overline{T_\lambda f_{k'}}\:w_3\\
&\lesssim\int_{\mathbb{R}}\sum_{k\in\mathbb{Z}}|T_\lambda f_k|^2\:w_3.
\end{aligned}
\end{eqnarray}
Now let $\Psi\in\mathcal{S}(\mathbb{R})$ be such that
\begin{eqnarray*}
\widehat{\Psi}(\xi) & = & \left\{
\begin{array}{c}
1 \quad\mbox{if}\quad |\xi| \leq 2 \\
0 \quad\mbox{if}\quad |\xi| \geq 4
\end{array} \right.,
\end{eqnarray*}
and define $\Psi_{L}, \Psi_{L,k}\in\mathbb{S}(\mathbb{R})$ as we did $W_L$ and $W_{L,k}$ previously.
Since $\widehat{\Psi}_{L,k}(\xi)=1$ when $\xi\in\supp\;\widehat{f}_{k}$, by \eqref{gotorth} followed by \eqref{mol2} of Observation \ref{elobs}
we have
\begin{equation}\label{juncture}
\int_{\mathbb{R}}|T_\lambda f|^2w\lesssim
\int_{\mathbb{R}}\sum_{k\in\mathbb{Z}}|T_\lambda f_k|^2\:w_3
\leq
\sum_{k\in\mathbb{Z}}\|T_\lambda\Psi_{\ell,\lambda}\|_1\int_{\mathbb{R}}|f|^2|T_\lambda\Psi_{\ell,\lambda}|*w_3.
\end{equation}
Our next lemma tells us that provided the scale $L$ is small enough, the functions $|T_\lambda\Psi_{L,k}|$ satisfy estimates similar to those of $\Psi_L$ uniformly in $k$.
\begin{lemma}\label{heart} If $L=2^{-p/(\ell-1)}\lambda^{1/\ell}$ and $|k|\sim 2^p\lambda^{1/\ell}/L$ then
\begin{eqnarray*}
|T_\lambda\Psi_{L,k}(x)| \lesssim \lambda^{-\frac{1}{\ell}}2^{-\frac{p(\ell-2)}{2(\ell-1)}}
H_L(x)
\end{eqnarray*}
where for each $N\in\mathbb{N}$, the function $H_L$ satisfies
\begin{eqnarray*}
H_L(x) \lesssim \frac{L}{(1+
L|x|)^N}
\end{eqnarray*}
for all $x\in\mathbb{R}$. The implicit constants depend on at most $\ell$, and finitely many of the $A_j$.
\end{lemma}
\begin{proof}
For $k\in\mathbb{Z}$ and $y\in\mathbb{R}$ let $h_k(y)=\phi(y)-\lambda^{-1}kLy$ and observe that
$$T_\lambda\Psi_{L,k}(x)=e^{ikLx}\int_\mathbb{R} e^{i\lambda h_k(y)}
\psi(y) \Psi_L(x-y)dy.$$
Now, $y_0$ is a stationary point of the phase $h_k$ precisely when $\phi'(y_0)=\lambda^{-1}kL$. However, since $|k|\sim 2^p\lambda^{1/\ell}/L$ and $L=2^p\lambda^{1/\ell}$
we have \begin{equation}\label{small}
\lambda^{-1}|k|L\sim 1/L^{\ell-1},
\end{equation}
and since $|\phi'(y)|\sim |y|^{\ell-1}$ (by \eqref{upperlower}), we have
$|y_0|\sim 1/L$. As usual the implicit constants depend only on $\ell$ and finitely many of the $A_j$. We note that since $\phi^{(\ell)}\not=0$, there are at most $\ell-1$ such stationary points $y_0$ in the support of $\psi$.

Let $(\eta_j)_{j=0}^\infty$ be a smooth partition of unity on $\mathbb{R}$ with
$\supp \: \eta_j \subset \{ x \in \mathbb{R}: 2^{j-1} \leq |x| \leq 2^{j+1}\}$ for $j\geq 1$ and
$\supp \: \eta_0 \subset \{ x \in \mathbb{R}: |x|\leq 2\}$. For uniformity purposes we suppose that $(\eta_j)$ is constructed in the standard way from a fixed smooth bump function and taking differences. Define $(\eta_{L,j})_{j=0}^\infty$ by $\eta_{L,j}(x)=\eta_j(Lx)$. Clearly $(\eta_{L,j})_{j=0}^\infty$ forms a partition of unity on $\mathbb{R}$ with
$\supp \: \eta_{L,j} \subset \{ x \in \mathbb{R}: |x|\sim 2^j/L\}$ for $j\geq 1$ and
$\supp \: \eta_{L,0} \subset \{ x \in \mathbb{R}: |x|\lesssim 1/L\}$.
We may thus write
\begin{equation*}
T_\lambda\Psi_{L,k} =
\sum_{j=0}^\infty T_\lambda(\Psi_{L,k}\eta_{L,j}).
\end{equation*}
We now write
\begin{equation}\label{locglob}
|T_\lambda\Psi_{L,k}(x)| \leq \sum_{j:2^j \geq c
L|x|} |T_\lambda(\Psi_{L,k}\eta_{L,j})(x)| +  \sum_{j:2^j <c
L|x|} |T_\lambda(\Psi_{L,k}\eta_{L,j})(x)|,
\end{equation}
where $c$ is a positive constant depending on $A_\ell$ to be chosen later.
We shall prove the required bounds for each of the sums in \eqref{locglob} separately.

Fix
$x$ and suppose that $2^j \geq  cL|x|$,
then integrating by parts we have
\begin{eqnarray*}
\begin{aligned}
|T_\lambda(\Psi_{L,k}\eta_{L,j})(x)| &= \left|\int_\mathbb{R} \frac{d}{dy} \left( \int_{-\infty}^y
e^{i(\lambda\phi(z) - kLz)} \psi(z) dz \right) \Psi_{L,k}(x-y) \eta_{L,j}(x-y)dy\right| \\
&\leq \int_\mathbb{R} \left| \int_{-\infty}^y e^{i(\lambda\phi(z) - kLz)}
\psi(z) dz \right| \left|\frac{d}{dy}(\Psi_{L,k}(x-y) \eta_{L,j} (x-y))\right|dy.
\end{aligned}
\end{eqnarray*}
By Lemma \ref{asymptotics k hat}, we have the estimate
$$
\left| \int_{-\infty}^y e^{i(\lambda\phi(z) - kLz)} \psi(z) dz \right|
\lesssim
\lambda^{-\frac{1}{\ell}}2^{-\frac{p(\ell-2)}{2(\ell-1)}},
$$
uniformly in $k\sim 2^p\lambda^{1/\ell}/L$ and $y$, and so
\begin{equation*}
|T_\lambda(\Psi_{L,k}\eta_{L,j})(x)|\lesssim \lambda^{-\frac{1}{\ell}}2^{-\frac{p(\ell-2)}{2(\ell-1)}} 2^{-Nj}L
\end{equation*}
for any $N \in \mathbb{N}$. Thus
\begin{eqnarray*}
\begin{aligned}
\sum_{j:2^j\geq cL|x|}
|T_\lambda(\Psi_{L,k}\eta_{L,j})(x)|&\lesssim \lambda^{-\frac{1}{\ell}}2^{-\frac{p(\ell-2)}{2(\ell-1)}}L
\sum_{j:2^j\geq cL|x|} 2^{-Nj}\\
&\sim \lambda^{-\frac{1}{\ell}}2^{-\frac{p(\ell-2)}{2(\ell-1)}}
\frac{L}{(1+L|x|)^N}
\end{aligned}
\end{eqnarray*}
similarly uniformly.

We now suppose that $2^j <  cL|x|$.
Then
\begin{equation}\label{estimate for I}
|T_\lambda(\Psi_{L,k}\eta_{L,j})(x)| \leq \lambda^{-N} \int_\mathbb{R}|(D_k^*)^N (\psi(y) \Psi_{L}(x-y)
\eta_{L,j}(x-y))| dy
\end{equation}
for any $N \in \mathbb{N}$, where the differential operator $D_k^*$ is given by
\begin{eqnarray*}
D_k^*g(y) = \frac{d}{dy} \left(\frac{g(y)}{h_k'(y)}\right).
\end{eqnarray*}
An elementary induction argument shows that $|(D_k^*)^Ng|$ is bounded by a sum of terms (the number of which depending only on $N$) of the form
\begin{equation}\label{action of D star}
\frac{|g^{(r)}|}{|h_k'|^n}\prod_{j=2}^{\ell-1}|\phi^{(j)}|^{m_j},
\end{equation}
uniformly in $k$,
where the indices $m_j$, $n$ and $r$ satisfy $$(\ell-1)n - \sum_{j=2}^{\ell-1} m_j(\ell-j) + r \leq \ell N,$$ $N \leq n \leq 2N$ and $0 \leq r \leq N$. Here only low derivatives of $\phi$ feature since we have used the bound $\|\phi^{(j)}\|_\infty\leq A_j$ for $j\geq\ell$.

Now, provided the constant $c>0$ in \eqref{locglob} is chosen sufficiently small (depending on at most $A_\ell$) then $|y|\sim |x|$ and by \eqref{small},
$$|h_k'(y)|=|\phi'(y)-\lambda^{-1}kL|\sim |y|^{\ell-1}\sim |x|^{\ell-1}$$ on the support of the integrand in \eqref{estimate for I}. Recalling also that $|\phi^{(j)}(y)|\sim |y|^{\ell-j}$ for all $2\leq j\leq \ell-1$, the expression \eqref{action of D star} is therefore comparable to
$$
\frac{|g^{(r)}(y)||y|^{\sum_{j=2}^{\ell-1}m_j(\ell-j)}}{|y|^{(\ell-1)n}}\lesssim |g^{(r)}(y)||x|^{r-\ell N}
$$
for all $y$ in the support of the integrand in \eqref{estimate for I}. Hence for $2^j<cL|x|$ and any $M\in\mathbb{N}$,
\begin{eqnarray*}
\begin{aligned}
|T_\lambda(\Psi_{L,k}\eta_{L,j})(x)|&\lesssim \lambda^{-N}|x|^{r-\ell N}\sum_{r=0}^N\int_{\mathbb{R}}\left|\left(\frac{d}{dy}\right)^r(\psi(y) \Psi_{L}(x-y)\eta_{L,j}(x-y))\right|dy\\
&\lesssim \lambda^{-N}|x|^{r-\ell N}\sum_{r=0}^N L^{r+1}2^{-jM}2^j/L\\
&=2^{-(M-1)j}(\lambda^{-\frac{1}{\ell}}L)^{\ell(N-1/2)}\lambda^{-\frac{1}{\ell}}2^{-\frac{p(\ell-2)}{2(\ell-1)}}\sum_{r=0}^N\frac{L}{(L|x|)^{\ell N-r}}\\
&\lesssim 2^{-(M-1)j}\lambda^{-\frac{1}{\ell}}2^{-\frac{p(\ell-2)}{2(\ell-1)}}\frac{L}{(1+L|x|)^{N}},
\end{aligned}
\end{eqnarray*}
since $\lambda^{-1/\ell}L=2^{-{p}/{(\ell-1)}}\leq 1$.
Thus
$$\sum_{j:2^j <c
L|x|} |T_\lambda(\Psi_{L,k}\eta_{L,j})(x)|\lesssim \lambda^{-\frac{1}{\ell}}2^{-\frac{p(\ell-2)}{2(\ell-1)}}\frac{L}{(1+L|x|)^{N}},$$ completing the proof of the lemma.
\end{proof}
If we let $w_4 = \lambda^{-\frac{2}{\ell}}2^{-\frac{p(\ell-2)}{\ell-1}}H_L * w_3$, then by \eqref{juncture} and Lemma \ref{heart} we conclude that
\begin{eqnarray*}
\int_\mathbb{R} |T_\lambda f(x)|^2 w(x)dx \lesssim \int_\mathbb{R}
\sum_k |f_k(x)|^2 w_4(x)dx.
\end{eqnarray*}
Now on applying Lemma \ref{LP 1}, our weighted estimate for
$T_\lambda$ becomes
\begin{eqnarray}\label{intermediate weighted}
\int_\mathbb{R} |T_\lambda f(x)|^2 w(x)dx \lesssim \int_\mathbb{R}
|f(x)|^2 w_5(x)dx
\end{eqnarray}
where $w_5 = |W_L| \ast w_4$.

In order to complete the proof of Proposition \ref{frequency restricted} for $\supp\:\widehat{f}\subseteq \mathcal{A}_p$ it remains to show that
\begin{equation}\label{introducing max function}
w_5(x) \lesssim
M\mathcal{M}_{\ell,\lambda}Mw(x).
\end{equation}
Since $w_5=\lambda^{-\frac{2}{\ell}}2^{-\frac{p(\ell-2)}{\ell-1}}|W_L|*H_L*\Theta_L*w_2$, we have $$w_5\lesssim \lambda^{-\frac{2}{\ell}}2^{-\frac{p(\ell-2)}{\ell-1}}Mw_2,$$ where $M$ denotes the Hardy--Littlewood maximal function.

By translation-invariance it thus suffices to show that $$\lambda^{-\frac{2}{\ell}}2^{-\frac{p(\ell-2)}{\ell-1}}w_2(0)\lesssim \mathcal{M}_{\ell,\lambda}Mw(0)$$ with implicit constant uniform in $p$.
Now,
$$
w_2(0)=\sup_{|y|\leq (4A_1)^{-\frac{1}{\ell-1}}/L} \Phi_{2^{p}\lambda^{1/\ell}}*w(y),$$ and so on setting $r=2^{-p}\lambda^{-1/\ell}$ we obtain
$$
\lambda^{-\frac{2}{\ell}}2^{-\frac{p(\ell-2)}{\ell-1}}w_2(0)\lesssim
r^{\frac{\ell-2}{\ell-1}}\lambda^{-\frac{1}{\ell-1}}\sup_{|y|\leq (4A_1\lambda r)^{-\frac{1}{\ell-1}}}\Phi_{1/r}*w(y).$$
For each $N\in\mathbb{N}$ we may estimate
\begin{eqnarray}\label{ana}
\begin{aligned}
r^{\frac{\ell-2}{\ell-1}}\lambda^{-\frac{1}{\ell-1}}\Phi_{1/r}*w(y)&\lesssim
(\lambda r)^{-\frac{1}{\ell-1}}\int_{\mathbb{R}}\frac{w(x)}{(1+|x-y|/r)^N}dx\\
&\sim (\lambda r)^{-\frac{1}{\ell-1}}\int_{|x-y|\leq r}w(x)dx\\&\;\;\;\;\;\;\;\;\;\;\;\;\;\;\;+\sum_{k=1}^\infty 2^{-kN}(\lambda r)^{-\frac{1}{\ell-1}}\int_{|x-y|\sim 2^kr}w(x)dx.
\end{aligned}
\end{eqnarray}


If for $s>0$ we define the averaging operator $A_s$ by
\begin{eqnarray*}
A_sw(y) = \frac{1}{2s} \int_{|x-y| \leq s} w(x)dx,
\end{eqnarray*}
then we may write
\begin{eqnarray*}
\begin{aligned}
(\lambda r)^{-\frac{1}{\ell-1}}\int_{|x-y| \sim 2^kr} w(x)dx &=(\lambda r)^{-\frac{1}{\ell-1}}2^kr A_{2^kr}w(y) \\ &\lesssim 2^k(\lambda r)^{-\frac{1}{\ell-1}}\int_{|y-y'| \leq r} A_{2^{k+1}r}w(y') dy' \\ &\leq 2^k (\lambda r)^{-\frac{1}{\ell-1}}\int_{|y-y'| \leq r} Mw(y') dy',
\end{aligned}
\end{eqnarray*}
since $A_{2^kr}w(y) \lesssim A_{2^{k+1}r}w(y')$ if $|y-y'| \leq r$.
Hence by \eqref{ana} we have
\begin{eqnarray}\label{fernando}
\begin{aligned}
r^{\frac{\ell-2}{\ell-1}}\lambda^{-\frac{1}{\ell-1}}\Phi_{1/r}*w(y) &\lesssim (\lambda r)^{-\frac{1}{\ell-1}}\int_{|x-y|\leq r}w(x)dx\\& +  (\lambda r)^{-\frac{1}{\ell-1}}\int_{|y-y'| \leq r} Mw(y') dy'.
\end{aligned}
\end{eqnarray}
Note that since $r=2^{-p}\lambda^{-1/\ell}$ and $1 \leq 2^p \leq 4A_1 \lambda^{(\ell-1)/\ell}$ we have that $(4A_1 \lambda)^{-1} \leq r \leq \lambda^{-1/(\ell-1)}$.
It will be convenient to consider separately the two cases $\lambda^{-1} \leq r \leq \lambda^{-1/(\ell-1)}$ and $(4A_1 \lambda)^{-1} \leq r \leq \lambda^{-1}$.\footnote{By making $A_1$ larger if necessary, we may of course assume that $4A_1 \geq 1$.}

If $\lambda^{-1} \leq r \leq \lambda^{-1/(\ell-1)}$, then $(y,r) \in \Gamma_{\ell, \lambda}(0)$ and so the first and second terms in the right hand side of \eqref{fernando} are dominated by $\mathcal{M}_{\ell, \lambda}(w)(0)$ and $\mathcal{M}_{\ell,\lambda}(Mw)(0)$ respectively.

If $(4A_1 \lambda)^{-1} \leq r \leq \lambda^{-1}$ then by taking $r'=\lambda^{-1}$ we have
\begin{equation}\label{james}
(\lambda r)^{-\frac{1}{\ell-1}}\int_{|x-y|\leq r}w(x)dx \leq (4A_1)^\frac{1}{\ell-1} (\lambda r')^{-\frac{1}{\ell-1}}\int_{|x-y|\leq r'}w(x)dx
\end{equation}
and
\begin{equation}\label{george}
(\lambda r)^{-\frac{1}{\ell-1}}\int_{|y-y'| \leq r} Mw(y') dy' \leq (4A_1)^\frac{1}{\ell-1} (\lambda r')^{-\frac{1}{\ell-1}}\int_{|y-y'| \leq r'} Mw(y') dy'.
\end{equation}
Furthermore, since $|y| \leq (4A_1 \lambda r)^{-1/(\ell-1)}$ an elementary calculation reveals that $(y,r') \in \Gamma_{\ell,\lambda}(0)$, and consequently \eqref{james} and \eqref{george} are dominated by constant multiples of $\mathcal{M}_{\ell, \lambda}(w)(0)$ and $\mathcal{M}_{\ell,\lambda}(Mw)(0)$ respectively. In each case we obtain the bound
\begin{eqnarray*}
r^{\frac{\ell-2}{\ell-1}}\lambda^{-\frac{1}{\ell-1}}\Phi_{1/r}*w(y) \lesssim \mathcal{M}_{\ell,\lambda}(Mw)(0),
\end{eqnarray*}
as claimed, with implicit constant depending on (at most) $A_1$ and $\ell$.

This completes the proof of Proposition \ref{frequency restricted} for $1< 2^p<4A_1\lambda^{(\ell-1)/\ell}$, leaving only the case $p=0$ to consider. To deal with this we observe that when $p=1$, the second frequency decomposition \eqref{clever} is effectively vacuous since $L=2^{-p/(\ell-1)}\lambda^{1/\ell}\sim 2^p\lambda^{1/\ell}$. Given this, it is straightforward to verify that our analysis in the case $p=1$ is equally effective in the case $p=0$. We leave this to the reader.

\section{The proof of Proposition \ref{frequency restricted} for $2^p\geq 4A_1\lambda^{(\ell-1)/\ell}$}\label{mainsecinf}
Suppose that the Fourier support of $f$ is contained in $\mathcal{A}_p$ for some $p$ with $2^p\geq 4A_1\lambda^{(\ell-1)/\ell}$. As we shall see, the rapid decay in \eqref{rapidap} means that such terms may be viewed as error terms in the sense that \eqref{two weighted restricted} holds with $\mathcal{M}_{\ell,\lambda}$ replaced with a much smaller operator.

Now, let $\Phi_{2^p\lambda^{1/\ell}}$ be constructed in the usual way by dilating a fixed Schwartz function such that $\widehat{\Phi}_{2^p\lambda^{1/\ell}}(\xi)=1$ for $\xi\in\mathcal{A}_p$. By \eqref{mol2} of Observation \ref{elobs}, we have
$$\int_{\mathbb{R}}|T_\lambda f|^2w\leq\|T_\lambda\Phi_{2^p\lambda^{1/\ell}}\|_1\int_{\mathbb{R}}|f|^2|T_\lambda\Phi_{2^p\lambda^{1/\ell}}|*w.$$
Now, for each $k\in\mathbb{N}$, repeated integration by parts gives
\begin{eqnarray*}
\begin{aligned}
|T_\lambda\Phi_{2^p\lambda^{1/\ell}}(x)|&=\Bigl|\int_{\mathbb{R}}e^{ix\xi}\widehat{K}_\lambda(\xi)\widehat{\Psi}_{2^p\lambda^{1/\ell}}(\xi)d\xi\Bigr|\\
&\leq |x|^{-k}\int_{\mathbb{R}}\Bigl|\Bigl(\frac{d}{d\xi}\Bigr)^k\left(\widehat{K}_\lambda(\xi)\widehat{\Phi}_{2^p\lambda^{1/\ell}}(\xi)\right)\Bigr|d\xi.
\end{aligned}
\end{eqnarray*}
Using Lemma \ref{asymptotics k hat} and the fact that $|\xi|\sim 2^p\lambda^{1/\ell}\gtrsim \lambda$ for $\xi\in\supp\:\widehat{\Phi}_{2^p\lambda^{1/\ell}}$ we conclude that for each $k,N\in\mathbb{N}$,
$$|T_\lambda\Phi_{2^p\lambda^{1/\ell}}(x)|\lesssim\lambda^{-N}\frac{\lambda}{(1+\lambda|x|)^k}.$$
By repeating the analysis that leads to \eqref{introducing max function} it is now straightforward to verify that
$$
\|T_\lambda\Phi_{2^p\lambda^{1/\ell}}\|_1|T_\lambda\Phi_{2^p\lambda^{1/\ell}}|*w\lesssim\mathcal{M}_{\ell,\lambda}Mw(0),
$$ uniformly in all parameters. This completes the proof of Proposition \ref{frequency restricted}.

\section{The proof of Proposition \ref{maximal bounds}}\label{corsection}

Since our proof will use a Hardy space estimate and complex interpolation, we begin by bounding $\mathcal{M}_{\ell,\lambda}$ by a more ``regular" maximal operator.
For $x \in \mathbb{R}$ let
\begin{eqnarray*}
\widetilde{\Gamma}_{\ell,\lambda}(x)=\{ (y,r) : 0 < r \leq \lambda^{-\frac{1}{\ell}} \quad\mbox{and}\quad |y-x| \leq (\lambda r)^{-\frac{1}{\ell-1}} \},
\end{eqnarray*}
and observe that $\Gamma_{\ell,\lambda}(x)\subseteq\widetilde{\Gamma}_{\ell,\lambda}(x)$.
Let $P$ be a nonnegative compactly supported smooth bump function which is positive on $[-1,1]$, and for each $r>0$, let $P_r(x)=\frac{1}{r}P(\frac{x}{r})$.
Then for any weight function $w$,
\begin{eqnarray*}
\mathcal{M}_{\ell,\lambda}w(x) \lesssim \widetilde{\mathcal{M}}_{\ell,\lambda}w(x) := \sup_{(y,r) \in \widetilde{\Gamma}_{\ell,\lambda}(x)} r(\lambda r)^{-\frac{1}{\ell-1}} |P_r \ast w(y)|.
\end{eqnarray*}
Thus in order to prove Proposition \ref{maximal bounds} it suffices to show that
\begin{eqnarray*}
\|\widetilde{\mathcal{M}}_{\ell, \lambda}\|_{(\frac{\ell}{2})' \rightarrow
(\frac{\ell}{2})'} \lesssim \lambda^{-\frac{2}{\ell}},
\end{eqnarray*}
which by scaling is equivalent to
\begin{eqnarray}\label{bound on M tilde}
\|\widetilde{\mathcal{M}}_{\ell,1}\|_{(\frac{\ell}{2})' \rightarrow
(\frac{\ell}{2})'} \lesssim 1.
\end{eqnarray}
In order to prove (\ref{bound on M tilde}) we define
\begin{eqnarray*}
\mathcal{M}_{\ell}^\beta (\phi) (x) = \sup_{(y,r) \in
\widetilde{\Gamma}_{\ell}(x)} r^{\frac{\ell \beta}{\ell -1}} |P_r \ast \phi(y)|
\end{eqnarray*}
with $P_r$ as above. By
Stein's method of analytic interpolation (see
\cite{S1}), inequality (\ref{bound on M tilde}) can be obtained from the
estimates
\begin{eqnarray*}
\|\mathcal{M}_\ell^0 (\phi)\|_\infty \lesssim \|\phi\|_\infty
\end{eqnarray*}
and
\begin{eqnarray}\label{H1 bound}
\|\mathcal{M}_\ell^1 (\phi)\|_{L^1} \lesssim \|\phi\|_{H^1}.
\end{eqnarray}
The first estimate is elementary, and the second may be verified by
testing on atoms. Let $a$ be an $H^1$-atom with support interval $I$
(by translation invariance we may suppose that $I$ is centred at
the origin). For an atom $a$ as described above, we have the
pointwise bound
\begin{eqnarray*}
r^{\frac{\ell}{\ell-1}}|P_r \ast a(x)| \lesssim \left\{
\begin{array}{c}
r^{\frac{\ell}{\ell-1}}/|I|, \quad\mbox{if}\quad r \lesssim |I| \quad\mbox{and}\quad |x| \lesssim |I| \\
|I|/r^{2-\frac{\ell}{\ell-1}}, \quad\mbox{if}\quad r \gtrsim  |I| \quad\mbox{and}\quad |x| \lesssim r \\
0, \quad\quad\quad\quad\mbox{otherwise.}\quad\quad\quad\quad
\end{array} \right.
\end{eqnarray*}

By the pointwise bounds on $\mathcal{M}_\ell^1(a)$ that follow from the above estimate, one may conclude that
\begin{eqnarray*}
\|\mathcal{M}_\ell^1 (a)\|_{L^1} \lesssim 1
\end{eqnarray*}
which is sufficient to obtain (\ref{H1 bound}).

\subsubsection*{Acknowledgement} It is a pleasure to acknowledge the influence that Tony Carbery has had on the content and perspective of this work.



\end{document}